\documentclass[11pt,reqno]{amsart}
\usepackage{amsmath,amsfonts,amssymb,amsthm,amscd,latexsym,euscript,verbatim,stmaryrd}

\usepackage{enumitem}
	\setenumerate[0]{label={\rm (\roman*)},leftmargin=*,itemsep=.2ex}

\usepackage[capitalize]{cleveref}
    \newcommand{\itemref}[2]{\cref{#1}\ref{#1:#2}}

\theoremstyle{plain}
\newtheorem{theorem}{Theorem}[section]
\newtheorem{proposition}[theorem]{Proposition}
\newtheorem{lemma}[theorem]{Lemma}
\newtheorem{corollary}[theorem]{Corollary}
\newtheorem*{question}{Question}
\newtheorem*{main}{Main Theorem}

\theoremstyle{definition}
\newtheorem{remark}[theorem]{Remark}
\newtheorem{definition}[theorem]{Definition}
\newtheorem{example}[theorem]{Example}

\newcommand{\tol}{\tfrac{2}{\eta}}
\newcommand{\taus}[1]{\llbracket #1 \rrbracket}
\newcommand{\eqv}[1]{\;\equiv_{(#1)}\,}
\newcommand{\F}{\mathbb{F}}
\newcommand{\Z}{\mathbb{Z}}
\newcommand{\I}{\mathcal{I}}
\newcommand{\dash}{\nobreakdash-\hspace{0pt}}

\DeclareMathOperator{\Aut}{Aut}
\DeclareMathOperator{\ad}{ad}
\DeclareMathOperator{\Sym}{Sym}

\newenvironment{case}[1]{{\medskip\bfseries\noindent Case #1}.\em}{\unskip\smallskip}
\newenvironment{step}[1]{{\medskip\bfseries\noindent Step #1}.\em}{\unskip\smallskip}

\numberwithin{equation}{section}
\begin{document}

\title{Primitive 4-generated axial algebras of Jordan type}

\author[Tom De Medts, Louis Rowen, Yoav Segev]{Tom De Medts\quad Louis Rowen\quad Yoav Segev}
%\author{Tom De Medts}
%\author{Louis Rowen}
%\author{Yoav Segev}
\address{Tom De Medts\\
        Department of Mathematics: Algebra and Geometry\\
        Ghent University\\
        Krijgslaan 281 – S25, 9000 Gent, Belgium}
\email{tom.demedts@ugent.be}

\address{Louis Rowen\\
         Department of Mathematics\\
         Bar-Ilan University\\
         Ramat Gan\\
         Israel}
\email{rowen@math.biu.ac.il}
\address{Yoav Segev \\
         Department of Mathematics \\
         Ben-Gurion University \\
         Beer-Sheva 84105 \\
         Israel}
\email{yoavs@math.bgu.ac.il}
%\thanks{$^*$The first author was supported by the Israel Science Foundation grant 1623/16}

\date{\today}

\begin{abstract}
    We show that primitive $4$-generated axial algebras of Jordan type are at most 81-dimensional.
\end{abstract}

\maketitle

\section{Introduction}

Axial algebras were introduced in 2015 by Jonathan Hall, Felix Rehren and Sergey Shpectorov \cite{HRSdef}.
They are non-associative commutative algebras generated by \emph{axes}, i.e., idempotents for which the left multiplication operator is semisimple and such that the resulting eigenspaces multiply according to a given \emph{fusion law} (see \S 2 for precise definitions).

In the easiest interesting case, these multiplication operators admit precisely $3$ eigenvalues $0$, $1$ and $\eta$.
A typical example is provided by \emph{Jordan algebras}, where each idempotent gives rise to a \emph{Peirce decomposition} of the algebra. In this case, we have $\eta = \tfrac{1}{2}$, and the fusion law is the following.

\begin{table}[ht!]
	\renewcommand{\arraystretch}{1.5}
	\[\begin{array}{c|ccc}
		* & 1 & 0 & \eta \\ \hline
		1 & \{1\} & \emptyset & \{\eta\} \\
		0 & \emptyset & \{0\} & \{\eta\} \\
		\eta & \{\eta\} & \{\eta\} & \{1,0\}
	\end{array}\]
	\caption{The Jordan fusion law $\Phi(\eta)$}
	\label{tb:jordan}
\end{table}

We call the axial algebras with a fusion law $\Phi(\eta)$ \emph{axial algebras of Jordan type $\eta$}.
Other than Jordan algebras themselves, there are other interesting examples of axial algebras of Jordan type (for arbitrary values of $\eta \neq 0,1$), namely the \emph{Matsuo algebras} arising from $3$-transposition groups.
In this case, the dimension of the algebra is equal to the size of the normal generating set of $3$-transpositions of the group. (See \cref{ex:matsuo} below for details.)

The classification of $3$-transposition groups has a long history (see \cite{CH, H} and the references therein).
It is a highly non-trivial fact that finitely generated $3$-transposition groups are finite.
In fact, this is a consequence of the classification of finite simple groups, and a direct proof of this fact would be very valuable.
(See \cite[Theorem (1.3), p.\@~153]{CH}.)

One possible approach for such a direct proof is precisely via the corresponding Matsuo algebras.
More generally, we ask the following question. (We refer to Definition \ref{def:PAJ} below for the precise meaning.)
\begin{question}
    Let $A$ be a primitive axial algebra of Jordan type.
    Assume that $A$ is generated by a finite set of axes.
    Can we conclude that $A$ is finite-dimensional?
\end{question}
Notice that, by Corollary \ref{cor:fg} below, a positive answer to this question would show, in particular, that finitely generated $3$-transposition groups are finite. In fact, for $\eta \neq \tfrac{1}{2}$, it is equivalent.

%Let $D:=\Delta^G$ be the  consider $\eta\neq\tfrac{1}{2},$ then, considering 

It is natural to try to answer this question for an increasing number of axes. For $2$-generated primitive axial algebras of Jordan type, this is almost trivial: such algebras are at most $3$-dimensional.
(In fact, much more can be said: \cite[Theorem 1.1]{HRS} gives a complete classification of such algebras.)

For $3$-generated algebras, this question was answered affirmatively in the recent paper \cite{GS}: such algebras are at most $9$-dimensional.

Our main result is the following.
\begin{main}
        Primitive $4$-generated axial algebras of Jordan type $\eta$ are at most $81$-dimensional,
				for any $\eta.$ Moreover, this result is best possible.
\end{main}

To go from $3$-generated to $4$-generated primitive axial algebras of Jordan type is a large step that required substantial new ideas. In fact, in our new setup, it is almost a triviality to recover the earlier result from \cite{GS} that such $3$-generated algebras are at most $9$-dimensional.
One of the key ideas is that we will almost never use the actual multiplication in the algebra, but instead, we use \emph{sequences of Miyamoto involutions} (see \cref{def:taus} below).
These sequences will allow us to formulate many ``rewriting rules'' that we can use to systematically deal with larger and larger expressions, until we eventually ``wrap up'' so that we can reduce every possible expression of length larger than~$6$. The precise meaning of this will be explained below and can be seen in \cref{thm:gamma}, which is a more detailed version of our Main Theorem.

It is worth pointing out that going to the next step, primitive $5$-generated axial algebras of Jordan type, is expected to be increasingly more difficult, because the upper bound of the dimension will be at least $3^{12} = 531441$. (In fact, this is our conjectured upper bound.) In addition, one of the examples (of dimension $306936$) arises from the largest sporadic Fischer group $\mathrm{Fi}_{24}$.

\section{Primitive axial algebras of Jordan type}

Throughout the paper, $\F$ will be a commutative field with $\operatorname{char} \F \neq 2$.
All our algebras will be commutative but non-associative\footnote{As usual, non-associative means ``not necessarily associative''.} $\F$-algebras.

For the definition of fusion laws and axial algebras, we rely on \cite{DMPSVC}.
\begin{definition}
    \begin{enumerate}
        \item 
            A \emph{fusion law} is a pair $(X, *),$ where $X$ is a set and $*$ is a map from $X \times X$ to $2^X$, where $2^X$ denotes the power set of $X$.
            A fusion law $(X, *)$ is called \emph{symmetric} if $x * y = y * x$ for all $x,y \in X$.
        \item
            The \emph{Jordan fusion law} is the fusion law with $X = \{ 0, 1, \eta \}$ (where $\eta$ is just a symbol) and with $*$ given by \cref{tb:jordan} above.
    \end{enumerate}
\end{definition}
\begin{definition}
    Let $\Phi = (X, *)$ be a fusion law.
    \begin{enumerate}
        \item
          A \emph{$\Phi$-decomposition} of an algebra $A$ is a direct sum decomposition $A = \bigoplus_{x \in X} A_x$ (as vector spaces) such that $A_x A_y \subseteq A_{x * y}$ for all $x,y \in X$, where $A_{Y} := \bigoplus_{y \in Y} A_y$ for all $Y \subseteq X$.
        \item
            A \emph{$\Phi$-decomposition algebra} is a triple $(A, \I, \Omega)$ 
            where $A$ is an $\F$-algebra,
            $\I$ is an index set and $\Omega$ is a tuple of $\Phi$-decompositions of $A$ indexed by $\I$.
            In other words, for each $i \in \I$, we have a corresponding $\Phi$\dash decomposition $A = \bigoplus_{x \in X} A^{(i)}_x$ of the algebra $A$.
    \end{enumerate}
\end{definition}
\begin{definition}\label{def:PAJ}
    Let $\Phi = (X, *)$ be a fusion law with $1 \in X \subseteq \F$.
    \begin{enumerate}
        \item For each $a \in A$, we write $\ad_a$ for the left multiplication by $a$, i.e., $\ad_a \colon A \to A \colon x \mapsto ax$.
        \item An element $a \in A$ is called a \emph{$\Phi$-axis} if it is idempotent (i.e., $a^2 = a$) and the decomposition of $A$ into the eigenspaces for $\ad_a$ is a $\Phi$\dash decomposition.
        \item The algebra $A$ is a \emph{$\Phi$-axial algebra} if it is generated by a set of $\Phi$-axes. This makes $A$ into a $\Phi$-decomposition algebra (with $\I$ identified with the given set of axes).
        \item A $\Phi$-axial algebra $A$ is \emph{primitive} if for each axis $a$ of the generating $\Phi$-axes of $A,$ the $1$-eigenspace $A^{(a)}_1$ is $1$-dimensional, i.e., is equal to $\F a$.
        \item An \emph{axial algebra of Jordan type $\eta$} is a $\Phi$-axial algebra for the fusion law $\Phi = \Phi(\eta)$ as in \cref{tb:jordan}.
    \end{enumerate}
\end{definition}

As we mentioned in the introduction, the two main sources of examples of axial algebras of Jordan type are (1) Jordan algebras, and (2) Matsuo algebras. We give some details.
\begin{example}
    Let $J$ be a Jordan algebra over $\F$, i.e., $J$ is a unital commutative non-associative algebra such that $a^2(ab) = a(a^2b)$ for all $a,b \in J$. If $e \in J$ is an idempotent, then it is an axis for the Jordan fusion law $\Phi(\tfrac{1}{2})$; this is the famous \emph{Peirce decomposition} for Jordan algebras (see, e.g., \cite[Chapter III]{Jac}).
    In particular, if $J$ is generated by idempotents, then it is an axial algebra of Jordan type $\tfrac{1}{2}$.
\end{example}
\begin{example}\label{ex:matsuo}
    Let $(G, D)$ be a \emph{$3$-transposition group}, i.e., $G$ is a group and $D \subseteq G$ is a generating set of involutions, closed under conjugation in $G$, such that the product of any two elements in $D$ has order at most $3$.
    Let $\eta \in \F \setminus \{ 0,1 \}$ be arbitrary.
    Then the \emph{Matsuo algebra} $M_\eta(G,D)$ is the algebra with basis $D$, with multiplication given by
    \[
        de := \begin{cases}
            e & \text{ if } d = e \\
            0 & \text{ if } o(de) = 2 \\
            \tfrac{\eta}{2} ( d + e - f ) & \text{ if } o(de) = 3, \text{ where } f = d^e = e^d \text{ in } G .
        \end{cases}
    \]
    By \cite[Theorem 6.5]{HRS}, $M_\eta(G,D)$ is a primitive axial algebra of Jordan type $\eta$.
\end{example}

Axial algebras of Jordan type, and more generally any type of decomposition algebras where the fusion law admits a $\Z/2$-grading, admit many involutory automorphisms, the so-called \emph{Miyamoto involutions}.
\begin{definition}
    \begin{enumerate}
        \item A \emph{$\Z/2$-grading} of a fusion law $(X, *)$ is a map $\theta \colon X \to \Z/2$ such that $x * y \subseteq \theta^{-1}(\theta(x) + \theta(y))$ for all $x,y \in X$. For instance, the Jordan fusion law from \cref{tb:jordan} is $\Z/2$-graded with $\theta(0) = \theta(1) = 0$ and $\theta(\eta) = 1$.
        \item If $(A, \I, \Omega)$ is a $\Phi$-decomposition algebra for a $\Z/2$-graded fusion law $(X,*)$, then for each $i \in \I$, we define a \emph{Miyamoto involution}
        \[ \tau_i \colon A \to A \colon a_x \mapsto (-1)^{\theta(x)} a_x, \quad \text{ when } a_x \in A^{(i)}_x . \]
        In other words, $\tau_i$ fixes the $0$-graded elements and negates the $1$-graded elements with respect to the $i$-th decomposition of $A$.
    \end{enumerate}
\end{definition}

\cref{cor:fg} below is an important motivation for the main result of our paper.

\begin{proposition}\label{pr:fg}
    Let $(G,D)$ be a $3$-transposition group. The following are equivalent:
    \begin{enumerate}[label={\rm (\alph*)}]
        \item\label{pr:fg:G} $G$ is finite.    
        \item\label{pr:fg:D} $D$ is finite.
        \item\label{pr:fg:M} $M_\eta(G,D)$ is finite-dimensional.
    \end{enumerate}
\end{proposition}
\begin{proof}
    Of course, \ref{pr:fg:G} implies \ref{pr:fg:D}, and \ref{pr:fg:D} and \ref{pr:fg:M} are equivalent because $M_\eta(G,D)$ has dimension~$|D|$.
    In particular, the dimension of $M_\eta(G,D)$ is independent of the choice of the base field $\F$ and of $\eta \in \F$, so to show that \ref{pr:fg:M} implies \ref{pr:fg:G}, we may assume that $\F$ is a finite field.
    
    Then $A := M_\eta(G,D)$ is finite.
    By \cite[p.~325]{DR} (which relies on \cite[p.\@~92, Example (4)]{Asch-3}), $G/Z(G)$ is embedded in $\Aut(A)$, so $G/Z(G)$ is a finite group.
    By a theorem of Schur, \cite[(33.9), p.~168]{Asch}, the derived subgroup $G'$ is finite.
    Since $G/G'$ is an abelian group generated by a finite number of involutions, it is finite, so we conclude that $G$ is finite.
\end{proof}

\begin{corollary}\label{cor:fg}
    The following are equivalent:
    \begin{enumerate}[label={\rm (\alph*)}]
        \item\label{cor:fg:a} Every finitely generated $3$-transposition group $(G,D)$ is finite.
        \item\label{cor:fg:b} Every primitive axial algebra $A$ of Jordan type $\eta\neq\tfrac{1}{2}$ generated by a finite set of axes $X$ is finite-dimensional.
    \end{enumerate}
\end{corollary}
\begin{proof}
    \begin{itemize}[itemindent=2ex]
        \item[\ref{cor:fg:a}\,$\Rightarrow$\,\ref{cor:fg:b}] 
            Let $A$ and $X$ be as in \ref{cor:fg:b}.  For $x\in X,$ let $\tau_x$ be the Miyamoto involution associated with $x.$  By \cite[Theorem (5.4), p.~105]{HRS}, the group $G = \langle \tau_x\mid x\in X\rangle$ is a $3$-transposition group.  By the assumption, $G$ is finite. By \cite[Corollary (1.2), p.~81]{HRS}, $A$~is spanned by $\{x^g\mid x\in X, g\in G\},$ so $A$ is finite-dimensional.
        \item[\ref{cor:fg:b}\,$\Rightarrow$\,\ref{cor:fg:a}]
            Let $(G,D)$ be a finitely generated $3$-transposition group.
            Then $G$ is generated by a finite number of elements from $D$, hence the algebra $M_{\eta}(G,D)$ is finitely generated.
            Thus, by the assumption, it is finite-dimensional. \Cref{pr:fg} then tells us that $G$ is finite.
        \qedhere
\end{itemize}    
\end{proof}

In order to get an idea about the complexity of the primitive $4$-generated axial algebras of Jordan type, it is useful to look at the list of $4$-generated $3$\dash transposition groups first.
In particular, this will provide us with an example of such an algebra of dimension $81$, which is precisely the upper bound that we will obtain in our main result.

\begin{theorem}
    %[{\cite[Theorem (3.3)]{CH}}]
    Let $(G,D)$ be a $3$-transposition group generated by $4$ elements from $D$ (but not by less than $4$). Then its central type is one of the following:
    \begin{enumerate}
        \item $W(A_4)$, the Weyl group of type $A_4$ (with $|D| = 10$);
        \item $W(D_4)$, the Weyl group of type $D_4$ (with $|D| = 12$);
        \item $3^3 \colon \Sym(4)$ (with $|D| = 18$);
        \item $2^{1+6} \colon \mathrm{SU}_3(2)'$ (with $|D| = 36$);
        \item Hall's $3$-transposition group $[3^{10}] \colon 2$ (with $|D| = 81$) or its affine quotient $3^{3+3} \colon 2$ (with $|D| = 27$).
    \end{enumerate}
\end{theorem}
\begin{proof}
    The definition of central type, and the proof of this fact (together with the size of $D$ in each case) can be found in \cite[Proposition (4.2)]{HS95}, where the authors point out that this classification has been proven independently by Zara, Hall and Moori; the first written source seems to be Zara's (unpublished) thesis from 1984.
\end{proof}

The unique $3$-transposition group in this list attaining the upper bound $|D| = 81$ is particularly interesting because it arises as a $3$-transposition subgroup of the sporadic Fischer groups $\mathrm{Fi}_{23}$ and $\mathrm{Fi}_{24}$.
We give an explicit construction of the resulting Matsuo algebra, based on \cite[\S 4.1]{LB83}.
In fact, we had implemented this example on a computer to experiment with identities, which is how some of our ideas arose.
\begin{example}
    Let $D$ be the $4$-dimensional vector space over the field $\F_3$ (so $|D| = 81$).
    We first set
    \begin{multline*}
        (x_1, x_2, x_3, x_4) \bullet (y_1, y_2, y_3, y_4) \\
        := \bigl( x_1 + y_1,\ x_2 + y_2,\ x_3 + y_3,\ x_4 + y_4 + (x_1 y_2 - x_2 y_1) (x_3 - y_3) \bigr)
    \end{multline*}
    for all $x_i, y_i \in \F_3$.
    Next, we set
    \[ d * e := (d \bullet e) \bullet (d \bullet e) \]
    for all $d,e \in D$.
    For any $\eta \in \F \setminus \{ 0,1 \}$---recall that $\F$ is still our arbitrary base field of characteristic different from $2$---we now define an $\F$-algebra with basis $D$, and with multiplication given by
    \[
        de := \begin{cases}
            e & \text{ if } d = e \\
            \tfrac{\eta}{2} ( d + e - d * e ) & \text{ if } d \neq e .
        \end{cases}
    \]
    Then by combining \cite{LB83} with \cref{ex:matsuo}, we see that this is precisely the Matsuo algebra corresponding to Hall's $3$-transposition group $[3^{10}] \colon 2$.
\end{example}

\section{Method}

From now on, we assume that $A$ is a primitive axial algebra of Jordan type~$\eta$ generated by a finite set $S$ of axes.
\begin{definition}    
    For each $i \geq 0$, we set
    \[ S[i] := \langle \tau_{a_1} \tau_{a_2} \dotsm \tau_{a_\ell} (b) \mid \ell \leq i, a_1,\dots,a_\ell,b \in S \rangle . \]
    In particular, $S[0] = \langle S \rangle$, and the $S[i]$ form an ascending chain of subspaces of $A$.
\end{definition}

Our goal is to show that $S[n] = A$ for some $n$. The following proposition tells us that we can do this by showing that the ascending chain of the $S[i]$ stabilizes.
\begin{proposition}
    Assume that $S[n] = S[n+1]$ for some $n$.
    Then $A = S[n]$.
\end{proposition}
\begin{proof}
    Following \cite[p.\@~81]{HRS}, we define the \emph{closure} of the set $S$ of axes to be the smallest set $C$ of axes of $A$ containing $S$ such that for each $a \in C$, we have $\tau_a(C) \subseteq C$.
    In fact, $C = \{ \tau_{a_1} \tau_{a_2} \dotsm \tau_{a_\ell} (b) \mid \ell \geq 0, a_1,\dots,a_\ell,b \in S \}$; see, for instance, \cite[Lemma 3.5]{KMS}.
    It now suffices to observe that if $S[n] = S[n+1]$, then $S[n] = S[\ell]$ for all $\ell \geq n$, hence $S[n] = \langle C \rangle$.
    By \cite[Cor.~(1.2), p.\@~81]{HRS}, however, $A$ is spanned by $C$, and the result follows.
\end{proof}
From now on, when we refer to an arbitrary \emph{axis of $A$}, we will always mean an element of the closure $C$ of $S$ (which is indeed always an axis for the same fusion law).

\medskip

The following two definitions will play a crucial role.
\begin{definition}\label{def:taus}
    \begin{enumerate}
        \item 
            We let
            \[ \taus{a_1, a_2, \dots, a_\ell} := \tau_{a_1} \tau_{a_2} \dotsm \tau_{a_\ell} \]
            for all axes $a_1,\dots,a_\ell \in A$. % (Notice that we do not require $a_i \in S$.)
        \item 
            For all $x,y \in A$, we set
            \[ x \eqv{i} y \iff x - y \in S[i] . \]
            Notice that $x \eqv{i} y$ implies $x \eqv{j} y$ for all $j \geq i$, and also implies that $\taus{a_1,\dots,a_\ell}x \eqv{i+\ell} \taus{a_1,\dots,a_\ell}y$ for all $a_1,\dots,a_\ell \in S$.
    \end{enumerate}
\end{definition}

\begin{remark}
    The notation $\taus{a_1,\dots,a_\ell}$ will also be used when the $a_i$ are axes that are not necessarily contained in $S$. Some care is needed with the use of the equivalence relations $\eqv{i}$ in such a situation, as these relations are always meant with respect to the given generating set $S$.
\end{remark}

By \cite[Theorem 4.1]{HSS18}, primitive axial algebras of Jordan type always admit a (necessarily unique) normalized 
symmetric Frobenius form.
\begin{definition}
    \begin{enumerate}
        \item A bilinear form $( \cdot, \cdot) \colon A \times A \to \F$ is called a \emph{(normalized) Frobenius form} on $A$ if $(xy, z) = (x, yz)$ for all $x,y,z \in A$ and, in addition, $(a, a) = 1$ for each axis $a \in A$.
        \item It will be useful to introduce the notation
            \[ \epsilon_{x,y} := 1 - \tol (x,y) \]
            for all $x,y \in A$.
    \end{enumerate}
\end{definition}
\begin{proposition}\label{pr:tau-op}
    Let $a \in A$ be an axis and $x \in A$ be arbitrary. Then
    \[ \tau_a(x) = x + \tol (a,x) a - \tol ax . \]
\end{proposition}
\begin{proof}
    This is \cite[Lemma 3.3]{HSS18} combined with the statement from \cite[Theorem 4.1]{HSS18} that $(a, x) = \varphi_a(x)$.
%    However, the proof of \cite[Lemma 3.3]{HSS18} is rather complicated (it depends on the classification of the $2$-generated algebras) so we provide an easy direct proof instead.
%    
%    So decompose $x \in A$ as $x = x_0 + x_1 + x_\eta$ with respect to the eigenspaces for the axis $a$; then $\tau_a(x) = x_0 + x_1 - x_\eta$.
%    On the other hand, we have $ax = x_1 + \eta x_\eta$. Then
%    \[ \tau_a(x) - x + \tol ax = -2 x_\eta + \tol (x_1 + \eta x_\eta) = \tol x_1 . \]
%    Since $A$ is primitive, we can write $x_1 = \zeta a$ for some $\zeta \in \F$, so it remains to show that $\zeta = (a, x)$.
\end{proof}
\begin{remark}
    In \cite{HSS18}, their Lemma 3.3 is used, in fact, in the proof of their Theorem 4.1 (the existence of the Frobenius form).
    On the other hand, if we already \emph{assume} the existence of the Frobenius form to begin with, then there is an easy direct proof of \cref{pr:tau-op} by simply decomposing $x$ with respect to the eigenspaces for the axis $a$.
\end{remark}

%A crucial ingredient is the observation that for each axis $a \in A$, we have
%\[ \tau_a = \id + \tol \pi_a - \tol L_a , \]
%or written explicitly,
%\begin{equation}\label{eq:tau-op}
%    \tau_a(x) = x + \tol (a,x) a - \tol ax
%\end{equation}
%for all $x \in A$.
\Cref{pr:tau-op} has the following immediate but useful consequences.
\begin{corollary}\label{pr:collect}
    Let $a,b \in A$ be axes. Then:
    \begin{enumerate}
        \item \label{pr:collect:ab-ba}
            $\taus{a}b - \taus{b}a = \epsilon_{a,b} (b-a)$.
        \item \label{pr:collect:tau=L}
            If $a \in S$ and $x \in S[i]$, then $(-\tol)ax \eqv{0} \taus{a}x - x \eqv{i} \taus{a}x$.
    \end{enumerate}    
\end{corollary}
\begin{proof}
    \begin{enumerate}
        \item 
            By \cref{pr:tau-op}, we have
            \[
                \tau_a(b) - \tau_b(a) = \bigl( 1 - \tol (a,b) \bigr) (b-a) .
            \]
        \item
            This follows immediately from \cref{pr:tau-op}.
        \qedhere
    \end{enumerate}
\end{proof}
We recall the following important fact, which we will be using over and over again, often without explicitly mentioning it.
\begin{proposition}\label{pr:tautau}
    We have $\taus{a,b,a} = \taus{\tau_a(b)}$ for all axes $a,b \in A$.
\end{proposition}
\begin{proof}
    This follows from \cite[Lemma 5.1, p.~103]{HRS} and the fact that $\tau_a \in \Aut(A)$.
\end{proof}

The following result is a first instance of how useful it is.
\begin{proposition}\label{pr:aba}
    Let $a,b \in S$ and $x \in A$. Then:
    \begin{enumerate}
        \item\label{pr:aba:aba} $\taus{a,b,a}x \eqv{1} x - \tol \tau_a(b) x$.
        \item\label{pr:aba:aba-bab} $\taus{a,b,a}x - \taus{b,a,b}x \eqv{1} \epsilon_{a,b} (\taus{b}x - \taus{a}x)$. In particular, if $x \in S[i]$ for some $i \geq 0$, then $\taus{a,b,a}x \eqv{i+1} \taus{b,a,b}x$. 
        \item\label{pr:aba:baba} $\taus{b,a,b,a}x \eqv{2} \taus{a,b}x + \epsilon_{a,b} (x - \taus{b,a}x)$. In particular, if $x \in S[i]$ for some $i \geq 2$, then $\taus{b,a,b,a}x \eqv{i} \taus{a,b}x - \epsilon_{a,b} \taus{b,a}x$.
    \end{enumerate}
\end{proposition}
\begin{proof}
    \begin{enumerate}
        \item 
            We apply \cref{pr:tautau} to $x$ and use \cref{pr:tau-op} on the right-hand side to get
            \begin{equation}\label{eq:aba}
                \taus{a,b,a}x = x + \tol (\tau_a(b), x) \taus{a}b - \tol \tau_a(b) x .
            \end{equation}
            Since $\taus{a}b \in S[1]$, the result follows.
        \item
            Interchanging $a$ and $b$ in \ref{pr:aba:aba} and subtracting gives, using \itemref{pr:collect}{ab-ba},
            \[
                \taus{a,b,a}x - \taus{b,a,b}x \eqv{1} (-\tol) \epsilon_{a,b} (b-a)x .
            \]
            By \itemref{pr:collect}{tau=L}, however,
            \[ (-\tol) (bx - ax) \eqv{0} (\taus{b}x - x) - (\taus{a}x - x) = \taus{b}x - \taus{a}x , \]
            and the result follows.
        \item
            This follows immediately by applying $\tau_b$ on \ref{pr:aba:aba-bab}.
        \qedhere
    \end{enumerate}
\end{proof}

%%%%%%%%%%%%%%%%%%%%%%%%%%%%%%%%
\begin{lemma}\label{lem:aba}
%%%%%%%%%%%%%%%%%%%%%%%%%%%%%%%%%
Let $a,b,c \in S$. Then:
\begin{enumerate}
\item\label{lem:aba:1}
$\taus{a,b}a = \epsilon_{a,b} a + b - \epsilon_{a,b} \taus{a}b \eqv{0} - \epsilon_{a,b} \taus{a}b$.
% Further, if $\taus{a}b\ne b$, then we may take $\beta=1$.

\item\label{lem:aba:2}
$\taus{a,b,a}c = \alpha c - \alpha \taus{a}b + \taus{c,a}b$ where $\alpha = \epsilon_{\tau_a(b), c} \in F$.

\item\label{lem:aba:3}
% $\taus{a,b,c}a= \alpha a+\beta b+\gamma c+\delta \taus{a}b-\epsilon_{a,c}\taus{a,b}c+\taus{c,a}b$, for some $\alpha,\beta,\gamma,\delta\in F$.
$\taus{a,b,c}a \eqv{0} \delta \taus{a}b - \epsilon_{a,c}\taus{a,b}c + \taus{c,a}b$ \ for some $\delta \in F$.
%  $\delta = - \epsilon_{\tau_a(b),c} - \epsilon_{a,b} \epsilon_{a,c} \in F$.

% \item
% $\taus{a,b,c,d,b}a \eqv{2} \taus{b,d,c,b}a - \epsilon_{a, \tau_b(d)} \taus{a, b, c} d$.
% (MOVED)
%
\end{enumerate}
\end{lemma}
\begin{proof}
\begin{enumerate}
\item
By \itemref{pr:collect}{ab-ba},
\[
\taus{a,b}a=\taus{a}(\taus{b}a-\taus{a}b+\taus{a}b)=\epsilon_{a,b}\taus{a} (a-b)+b.
\]

\item
%By equation \eqref{eq:aba}, with $x=c$, we get
%\[
%\textstyle{\taus{a,b,a}c=c+\frac{2}{\lambda}(\tau_a(b),c)\taus{a}b-\frac{2}{\lambda}c\tau_a(b).}
%\]
%By equation \eqref{eq:tau-op} (with $c$ in place of $a$ and $\tau_a(b)$ in place of $x$),
%\[
%\textstyle{-\frac{2}{\lambda}c\tau_a(b)=\taus{c,a}b-\taus{a}b-\frac{2}{\lambda}(c,\tau_a(b))c.}
%\]
%This shows (ii).
Let $\alpha = \epsilon_{\tau_a(b), c}$.
By substituting $\tau_a(b)$ for $a$ and $c$ for $b$ in \itemref{pr:collect}{ab-ba}, we get
\[ \taus{\tau_a(b)}c - \taus{c,a}b = \alpha (c - \taus{a} b) . \]
The result now follows from \cref{pr:tautau}.

\item
By \itemref{pr:collect}{ab-ba}, we have
        \begin{align*}
            \taus{a,b,c}a
            &= \taus{a,b,a}c + \taus{a,b} \bigl( \taus{c}a - \taus{a}c \bigr) \\
            &= \taus{a,b,a}c + \epsilon_{a,c} \taus{a,b} (a - c),
        \end{align*}
so \ref{lem:aba:3} follows from \ref{lem:aba:1} and \ref{lem:aba:2}.
\qedhere
\end{enumerate}
\end{proof}

\section{Rewriting rules}

In this section, we will gradually build up ``rewriting rules'' that will allow us to simplify certain expressions.
As the length of the expressions increases, the proofs become more and more involved.

\begin{proposition}\label{pr:rules}
    Let $a,b,c,d \in S$. Then:
    \begin{enumerate}
        \item\label{pr:rules:a b} $\taus{a}b \eqv{0} \taus{b}a$.
        \item\label{pr:rules:ab a} $\taus{a,b}a \in S[1]$.
        \item\label{pr:rules:aba c} $\taus{a,b,a}c \eqv{1} \taus{c,a}b$.
        \item\label{pr:rules:abc a} $\taus{a,b,c}a \eqv{1} \taus{c,b}a - \epsilon_{a,c} \taus{a,b}c$
        \item\label{pr:rules:abac d} $\taus{a,b,a,c}d \eqv{1} \taus{c,d,c,a}b$.
        \item\label{pr:rules:abcdb a} $\taus{a,b,c,d,b}a \eqv{2} \taus{b,d,c,b}a - \epsilon_{a, \tau_b(d)} \taus{a, b, c} d \eqv{3} \taus{b,d,c,b}a$.
        \item\label{pr:rules:abcab d} $\taus{a,b,c,a,b}d \eqv{3} \taus{b,a,d,b,a}c$.
    \end{enumerate}
\end{proposition}
\begin{proof}
    \begin{enumerate}
        \item This follows from \itemref{pr:collect}{ab-ba}.
        \item By \ref{pr:rules:a b}, we have $\taus{a,b}a \eqv{1} \taus{a,a}b = b$. (Of course, this also follows from \itemref{lem:aba}{1}.)
        \item This follows from \itemref{lem:aba}{2}.
        \item This follows from \itemref{lem:aba}{3} and \ref{pr:rules:a b}.
        \item We have
            \[ \taus{a,b,a,c}d - \taus{c,d,c,a}b = \taus{\tau_a(b)} \tau_c(d) - \taus{\tau_c(d)} \tau_a(b) ,\]
            which is contained in $\langle \tau_a(b) - \tau_c(d) \rangle \leq S[1]$ by \itemref{pr:collect}{ab-ba}.
        \item We have
            \[ \taus{a,b,c,d,b}a = \taus{a, \tau_b(c), \tau_b(d)}a . \]
            Now let $S' = \{ a, \tau_b(c), \tau_b(d) \}$ and apply \ref{pr:rules:abc a} with respect to this set $S'$ in place of $S$. Notice that $S'[1] \leq S[2]$, because
            \begin{align*}
                &\taus{\tau_b(c)}a = \taus{b,c,b}a \in S[2] \quad \text{(by \ref{pr:rules:aba c})}, \\
                &\taus{\tau_b(c)}\tau_b(d) = \taus{b,c,b,b}d = \taus{b,c}d \in S[2] ,
            \end{align*}
            so we see that indeed $\taus{x}y \in S[2]$ for all $x,y \in S'$.
            Hence
            \begin{align*}
                \taus{a, \tau_b(c), \tau_b(d)}a
                &\eqv{2} \taus{\tau_b(d), \tau_b(c)}a - \epsilon_{a, \tau_b(d)} \taus{a, \tau_b(c)} \tau_b(d) \\
                &= \taus{b,d,c,b}a - \epsilon_{a, \tau_b(d)} \taus{a, b, c} d
            \end{align*}
            so we conclude that indeed
            \[ \taus{a,b,c,d,b}a \eqv{2} \taus{b,d,c,b}a - \epsilon_{a, \tau_b(d)} \taus{a, b, c} d \eqv{3} \taus{b,d,c,b}a . \]
        \item
            We start from
            \begin{align*}
                \taus{b,a,c,a,b}d
                &= \taus{\tau_b\tau_a(c)}d \\
                &= \taus{d} \tau_b\tau_a(c) + \epsilon_{d, \tau_b \tau_a(c)} (d - \tau_b \tau_a(c)) \\
                &= \taus{d,b,a}c + \epsilon_{d, \tau_b \tau_a(c)} (d - \taus{b,a}c) \\
                &\eqv{0} \taus{d,b,a}c - \epsilon_{d, \tau_b \tau_a(c)} \taus{b,a}c.
            \end{align*}
            In particular, $\taus{b,a,c,a,b}d \in S[3]$.
            Moreover, applying $\taus{b,a}$ to this equivalence yields
            \begin{align}
                \taus{b,a,b,a,c,a,b}d
                &\eqv{2} \taus{b,a,d,b,a}c - \epsilon_{d, \tau_b \tau_a(c)} \taus{b,a,b,a}c \notag \\
                &\eqv{2} \taus{b,a,d,b,a}c , \label{eq:babacabd}
            \end{align}
            where the last equivalence holds because by \ref{pr:rules:aba c} and \ref{pr:rules:abc a}, we have
            \[ \taus{b,a,b,a}c \eqv{2} \taus{b,c,a}b \in S[2] . \]
            
            We now apply \itemref{pr:aba}{baba} with $x = \taus{c,a,b}d \in S[3]$, which gives
            \begin{align*}
                \taus{b,a,b,a,c,a,b}d
                &\eqv{3} \taus{a,b,c,a,b}d - \epsilon_{a,b} \taus{b,a,c,a,b}d \\
                &\eqv{3} \taus{a,b,c,a,b}d .
            \end{align*}
            The claim follows by combining this with \eqref{eq:babacabd}.
        \qedhere
    \end{enumerate}
\end{proof}

For our next rewriting rule in \cref{pr:finalrule}, we first need the following lemma.
		
%%%%%%%%%%%%%%%%%%%%%%%%%%%%%%%%%%%%%%%
\begin{lemma}\label{lem:S'}
%%%%%%%%%%%%%%%%%%%%%%%%%%%%%%%%5
Let $a,b,c,d \in S$ and let $S'=\{a,b,c,\tau_a(d)\}$. Then $S'[3]\subseteq S[4]$.
\end{lemma}
\begin{proof}
Let $\taus{x,y,z}w$ be any element with $x,y,z,w \in S'$. 
Of course, if none of these four elements is equal to $\tau_a(d)$, then $\taus{x,y,z}w \in S[3] \subseteq S[4]$, and if all four elements are equal to $\tau_a(d)$, then $\taus{x,y,z}w = \tau_a(d) \in S[1] \subseteq S[4]$.

\begin{case}{1}
    Suppose that only one of these four elements is equal to $\tau_a(d)$.    
\end{case}

If $w=\tau_a(d)$, then $\taus{x,y,z}w = \taus{x,y,z,a}d \in S[4]$.
If $z=\tau_a(d)$, then, by \itemref{pr:rules}{aba c},
\[
\taus{x,y,z}w=\taus{x,y,a,d,a}w\in S[4].
\]
If $y=\tau_a(d)$, then, by \itemref{pr:rules}{abac d},
\[
\taus{x,y,z}w=\taus{x,a,d,a,z}w\eqv{2}\taus{x,z,w,z,a}d.
\]
If $x=a$ or $z=a$, then $\taus{x,y,z}w\in S[4]$.
If $w=a$, then $\taus{x,y,z}w\in S[4]$, by \itemref{pr:rules}{ab a}.
We may thus assume that $z=b$ and $w=c$.
If $x=b$, then we see that $\taus{x,y,z}w\in S[4]$.
If $x=c$, then
\[ \taus{x,y,z}w \eqv{2} \taus{c,b,c,b,a}d \in S[3], \]
by \itemref{pr:aba}{baba}.

If $x=\tau_a(d)$, then, assuming without loss that $y=b$,
\[
\taus{x,y,z}w=\taus{a,d,a,b,z}w .
\]
If $z=a$, then by \itemref{pr:rules}{aba c}, $\taus{x,y,z}w\in S[4]$.
Hence we may assume $z=c$, and by \itemref{pr:rules}{ab a} we may assume that $w=a$.
In this case, \itemref{pr:rules}{abc a} shows that $\taus{x,y,z}w\in S[4]$.

\begin{case}{2}
    Suppose that three of the four elements $x,y,z,w$ are equal to $\tau_a(d)$.
\end{case}

If $x = y = z = \tau_a(d)$, then of course $\taus{x,y,z}w = \taus{\tau_a(d)}w = \taus{a,d,a}w \in S[2]$.
For the other cases, we simply observe that
\begin{align*}
    &\taus{\tau_a(d), x, \tau_a(d)}\tau_a(d) = \taus{\tau_a(d), x}\tau_a(d) = \taus{a,d,a,x,a}d \in S[4], \\
    &\taus{\tau_a(d),\tau_a(d),x}\tau_a(d)=\taus{x,a}d\in S[2], \\
    &\taus{x,\tau_a(d),\tau_a(d)}\tau_a(d)=\taus{x,a}d\in S[2].
\end{align*}

\begin{case}{3}
    Exactly two of the four elements $x,y,z,w$ are equal to $\tau_a(d)$.
\end{case}

We have
\begin{align*}
    &\taus{\tau_a(d),\tau_a(d),z}w=\taus{z}w\in S[1], \\
    &\taus{x,\tau_a(d),\tau_a(d)}w=\taus{x}w\in S[1], \text{ and} \\
    &\taus{x,y,\tau_a(d)}\tau_a(d)=\taus{x,y,a}d\in S[3].
\end{align*}
If $x = w = \tau_a(d)$ and $y,z \in \{ a,b,c \}$, then, by \ref{pr:rules:abcdb a},
\[ 
\taus{\tau_a(d), y, z}\tau_a(d) = \taus{a,d,a,y,z,a}d\eqv{4}\taus{a,a,z,y,a}d\in S[3]. 
\]
Next, if $x = w = \tau_a(d)$ and $y,z \in \{ a,b,c \}$, then,
by Lemma \ref{lem:aba}(ii) (with $\tau_a(d)$ in place of $a$), $\taus{\tau_a(d), y, \tau_a(d)}w\in S[4]$.

Finally, if $y = w = \tau_a(d)$ and $x,z \in \{ a,b,c \}$, then, by Lemma \ref{lem:aba}(i)
\[ 
\taus{x, \tau_a(d), z}\tau_a(d)\in S[4].\qedhere
\]
\end{proof}

The following corollary will play an important role in the proof of \cref{pr:keyrule}.
\begin{corollary}\label{co:T4 in S6}
    Let $a,b,c,d \in S$ and let $T = \{ a, \tau_a(b), \tau_a(c), d \}$.
    Then $T[4] \subseteq S[6]$.
\end{corollary}
\begin{proof}
    Let $S'=\{a,b,c,\tau_a(d)\}$ as in \cref{lem:S'} and notice that $T = \tau_a(S')$, i.e., $T$ is obtained from $S'$ by applying
		$\tau_a$ on each element.  By \cref{pr:tautau}, for all $x_1,\dots,x_k,y\in S'$ we have
		\[
		\taus{\tau_a(x_1),\dots,\tau_a(x_k)}\tau_a(y)=\taus{a,x_1,\dots,x_k,a}\tau_a(y)=\tau_a(\taus{x_1,\dots,x_k}y),
		\]
%    Since $\tau_a$ is an automorphism of the algebra $A$, this implies that 
so $T[i] = \tau_a(S'[i])$ for all $i$.
    
    By \cref{lem:S'}, we have $S'[3] \subseteq S[4]$. Now
    \begin{align*}
        S'[4]
        &= \taus{a} S'[3] \cup \taus{b} S'[3] \cup \taus{c} S'[3] \cup \taus{\tau_a(d)} S'[3] \\
        &\subseteq S[5] \cup \taus{a,d,a} S[4] ,
    \end{align*}
    and hence
    \[ T[4] = \taus{a} S'[4] \subseteq \taus{a} S[5] \cup \taus{d,a} S[4] \subseteq S[6] . \qedhere \]
\end{proof}

\begin{proposition}\label{pr:finalrule}
    Let $a,b,c,d \in S$. Then
    \[ \taus{a,b,c,a,b,c}d \eqv{4} \taus{b,c,a,b,c,a}d \eqv{4} \taus{c,a,b,c,a,b}d . \]
\end{proposition}
\begin{proof}
    Let $S' = \{a,b,c,\tau_a(d)\}$.
    By \cref{lem:S'}, we have $S'[3] \subseteq S[4]$.
    We can thus apply \itemref{pr:rules}{abcab d} with respect to $S'$ to get
    \[ \taus{c,b,\tau_a(d),c,b}a \eqv{4} \taus{b,c,a,b,c}\tau_a(d), \]
    hence
    \begin{equation}\label{eq:pf1} 
        \taus{c,b,a,d,a,c,b}a \eqv{4} \taus{b,c,a,b,c,a}d .
    \end{equation}
    On the other hand, we apply $\taus{c,b,a,d}$ to the equivalence in \itemref{lem:aba}{3} (with $b$ and $c$ interchanged) to get
    \begin{multline*}
        \taus{c,b,a,d,a,c,b}a \\ \eqv{4}\delta\taus{c,b,a,d,a}c-\epsilon_{a,b}\taus{c,b,a,d,a,c}b+\taus{c,b,a,d,b,a}c
    \end{multline*}
    for some $\delta \in F$.
    Now $\taus{c,b,a,d,a}c\in S[4]$ by \itemref{pr:rules}{aba c}.
    Also, by \itemref{pr:rules}{abac d} and \itemref{pr:aba}{baba}, we have
    \[ \taus{c,b,a,d,a,c}b\eqv{3}\taus{c,b,c,b,c,a}d\in S[4] . \]
    Thus, by \itemref{pr:rules}{abcab d},
    \begin{equation}\label{eq:pf2}
        \taus{c,b,a,d,a,c,b}a\eqv{4}\taus{c,b,a,d,b,a}c\eqv{4}\taus{c,a,b,c,a,b}d.
    \end{equation}
    Combining \eqref{eq:pf1} and \eqref{eq:pf2}, we see that
    \[ \taus{b,c,a,b,c,a}d\eqv{4}\taus{c,a,b,c,a,b}d. \]
    It now suffices to cyclically permute $a$$,b$$,c$ to also get the other equivalence.
\end{proof}

We now come to the final and most challenging rewriting rule, which will effectively put a bound on the dimension of $4$-generated primitive axial algebras of Jordan type.

\begin{proposition}\label{pr:keyrule}
    Let $a,b,c,d \in S$. Then
    $\taus{d,a,b,c,a,b,c}d\in S[6]$.
\end{proposition}
\begin{proof}
    Let
    \[ T=\{\tau_d(a),\tau_d(b),c,d\}. \]
    By \cref{co:T4 in S6}, we have $T[4] \subseteq S[6]$.
    By \cref{pr:finalrule} applied to $T$, this implies that 
    \begin{equation}\label{eq:main0}
        \taus{\tau_d(a),\tau_d(b),c,\tau_d(a),\tau_d(b),c}d\eqv{6}\taus{c,\tau_d(a),\tau_d(b),c,\tau_d(a),\tau_d(b)}d.
    \end{equation}
    We will proceed in two steps: We first show that
    \begin{equation}\label{eq:main1}
        \taus{\tau_d(a),\tau_d(b),c,\tau_d(a),\tau_d(b),c}d\eqv{6}\taus{d,a,c,b,a,c,b}d,
    \end{equation}
    and then we show that
    \begin{equation}\label{eq:main2}
        \taus{c,\tau_d(a),\tau_d(b),c,\tau_d(a),\tau_d(b)}d\in S[6].
    \end{equation}
    Interchanging the role of $b$ and $c$, it will then follow from \eqref{eq:main0}, \eqref{eq:main1} and~\eqref{eq:main2} that $\taus{d,a,b,c,a,b,c}d\in S[6]$.

    \begin{step}{1}
        Proof of \eqref{eq:main1}.
    \end{step}

    By \itemref{lem:aba}{1} applied on $\taus{d,c}d$, we have
    \begin{multline*}
        \taus{\tau_d(a),\tau_d(b),c,\tau_d(a),\tau_d(b),c}d \\
        \begin{aligned}
            &=\taus{d,a,b,d,c,d,a,b,d,c}d \\
            &=\taus{d,a,b,d,c,d,a,b}c + \epsilon_{c,d} \taus{d,a,b,d,c,d,a,b}d
        \end{aligned} \\
        - \epsilon_{c,d} \taus{d,a,b,d,c,d,a,b,d}c .
    \end{multline*}
    Now let $\gamma = -\epsilon_{c, \tau_d(b)}$; then by \itemref{pr:rules}{abcdb a} and \itemref{pr:rules}{abac d}, we have
    \begin{align*}
        \taus{d,a,b,d,c,d,a,b,d}c
        &\eqv{6} \taus{d,a,b,d,d,b,a,d}c + \gamma\taus{d,a,b,d,c,d,a}b \\
        &\eqv{0} \gamma\taus{d,a,b,d,c,d,a}b \\
        &\eqv{4} \gamma\taus{d,a,b,a,b,a,d}c \in S[6],
    \end{align*}
    by Proposition \ref{pr:aba}(iii).

    Also, by Proposition \ref{pr:rules}(iv), we have
    \begin{multline*}
        \taus{d,a,b,d,c,d,a,b}d \\
        \begin{aligned}
            &\eqv{6}\taus{d,a,b,d,c,b,a}d - \epsilon_{b,d}\taus{d,a,b,d,c,d,a}b\\
            &\eqv{6}\taus{d,a,b,d,c,b,d}a - \epsilon_{b,d}\taus{d,a,b,a,b,a,d}c\quad\text{(by \ref{pr:rules}\ref{pr:rules:a b} and \ref{pr:rules}\ref{pr:rules:abac d})} \\
            &\eqv{6}\taus{d,a,d,b,a,d,b}c - \epsilon_{b,d}\taus{d,a,b,a,b,a,c}d\quad \text{(by \ref{pr:rules}\ref{pr:rules:abcab d} and \ref{pr:rules}\ref{pr:rules:a b})} \\
            &\ \in S[6],  
        \end{aligned}
    \end{multline*}
    by \cref{pr:finalrule} and \itemref{pr:aba}{baba}.

    Finally, by \itemref{pr:aba}{aba-bab}, 
    \begin{multline*}
        \taus{d,a,b,d,c,d,a,b}c \\
        \begin{aligned}
            &\eqv{6} \taus{d,a,b,c,d,c,a,b}c \\
            &\eqv{6} \taus{d,a,b,c,d,b,a}c - \epsilon_{b,c}\taus{d,a,b,c,d,c,a}b \quad\text{(by \ref{pr:rules}\ref{pr:rules:abc a})} \\
            &\eqv{6}\taus{d,a,b,c,d,b,c}a - \epsilon_{b,c}\taus{d,a,b,a,b,a,c}d\quad\text{(by \ref{pr:rules}\ref{pr:rules:a b} and \ref{pr:rules}\ref{pr:rules:abac d})} \\
            &\eqv{5}\taus{d,a,c,b,a,c,b}d \quad\text{(by \ref{pr:rules}\ref{pr:rules:abcab d} and \ref{pr:aba}\ref{pr:aba:baba}).}
        \end{aligned}
    \end{multline*}
    This proves \eqref{eq:main1}.
 
    \begin{step}{2}
        Proof of \eqref{eq:main2}.
    \end{step}

    By Lemma \ref{lem:aba}(iii), there exists $\delta\in F$ with
    \begin{multline*}
        \taus{c,\tau_d(a),\tau_d(b),c,\tau_d(a),\tau_d(b)}d \\
        \begin{aligned}
            &= \taus{c,d,a,b,d,c,d,a,b}d\\
            &\eqv{6} \delta\taus{c,d,a,b,d,c,d}a - \epsilon_{b,d}\taus{c,d,a,b,d,c,d,a}b
        \end{aligned} \\
        + \taus{c,d,a,b,d,c,b,d}a .
    \end{multline*}
    Now by \itemref{pr:rules}{aba c}, $\taus{c,d,a,b,d,c,d}a\in S[6]$.
    By \itemref{pr:rules}{abac d} and \itemref{pr:aba}{baba}, we have
    \[ \taus{c,d,a,b,d,c,d,a}b\eqv{5}\taus{c,d,a,b,a,b,a,d}c\in S[6] . \]
    Finally, by \itemref{pr:rules}{abcab d} and \cref{pr:finalrule}, we also have
    \begin{align*}
        \taus{c,d,a,b,d,c,b,d}a
        &\eqv{6} \taus{c,d,a,d,b,a,d,b}c \\
        &\eqv{6} \taus{c,d,d,b,a,d,b,a}c
        = \taus{c,b,a,d,b,a}c \in S[6].
    \end{align*}
    This proves \eqref{eq:main2} and thus finishes the proof of this proposition.
\end{proof}

\section{$4$-generated primitive axial algebras of Jordan type}

We are now ready to prove our main result.
Although it requires some care to write down the proof, the hard work has already been done in \cref{pr:rules,pr:finalrule,pr:keyrule}.
\begin{theorem}\label{thm:gamma}
    Assume that $A$ is generated by a set $S = \{ a,b,c,d \}$ of $4$ axes. 
	Then $A = S[6]$ and $A$ is at most $81$-dimensional.
	
    More precisely, let $G$ be the group $\operatorname{Sym}(S)$ of all permutations of $S$.
    Define%
    \footnote{There is some obvious abuse of notation here: a priori, the group $G$ does not act on~$A$, so when we write an expression like $\{ \taus{a,b,c}d \}^G$, we really mean $\{ \taus{a^\rho, b^\rho, c^\rho}(d^\rho) \mid \rho \in G \}$.}
    \begin{align*}
        \Gamma_0 &= \{ a \}^G , \\
        \Gamma_1 &= \{ \taus{a}b \}^G , \\
        \Gamma_2 &= \{ \taus{a,b}c \}^G , \\
        \Gamma_3 &= \{ \taus{a,b,c}d \}^G , \\
        \Gamma_4 &= \{ \taus{a,b,a,c}d,\ \taus{a,b,c,a}d \}^G , \\
        \Gamma_5 &= \{ \taus{a,b,c,a,b}d \}^G , \\
        \Gamma_6 &= \{ \taus{a,b,c,a,b,c}d \}^G .
    \end{align*}
    Then for each $i \in \{ 1,\dots,6 \}$, we have $S[i] = \langle \Gamma_0, \dots, \Gamma_i \rangle$.
    In particular, $A = \langle \Gamma_0, \dots, \Gamma_6 \rangle$.
    
    Moreover, there is some redundancy in these spanning sets: The dimension of each of the $S[i]$ is at most $4$, $10$, $22$, $34$, $61$, $73$ and $81$, respectively.
\end{theorem}
\begin{proof}
    For each $i \leq 6$, let $T[i]$ be the subspace of $A$ spanned by $\Gamma_0,\dots,\Gamma_i$.
    Obviously, we have $T[i] \leq S[i]$ for each $i$.
    We will show recursively that for each $i \leq 6$, $S[i] = T[i]$, and that $S[7] = S[6]$.
    We will, at the same time, compute the maximal possible dimension of each~$T[i]$.
    Notice that for each~$i$, the subspace $S[i+1]$ is spanned by $S[i]$ and all elements obtained by applying the four operations $\taus{a}$, $\taus{b}$, $\taus{c}$ and $\taus{d}$ on the elements of $S[i]$.
    In order to go from $S[i] = T[i]$ to the next step $S[i+1]$, it will suffice, by $G$-symmetry, to apply these four operations on the given representative of the set~$\Gamma_i$.
    
    \begin{itemize}
        \item[$i=0$.]
            Obviously, $T[0] = \langle a,b,c,d \rangle = S[0]$, and $\dim S[0] \leq 4$.
        \item[$i=1$.]
            We have $\taus{a}a = a \in S[0]$, whereas applying any of the other three operations $\taus{b}$, $\taus{c}$, $\taus{d}$ on the representative $a \in \Gamma_0$ results in an element of $\Gamma_1$, so $S[1] \leq T[1]$.
            
            By \itemref{pr:rules}{a b}, we have $\taus{b}a \eqv{0} \taus{a}b$, so the $12$ possible elements of $\Gamma_1$ come in pairs that are linearly dependent modulo~$S[0]$. Hence $\dim S[1] \leq 4 + 12/2 = 10$.
        \item[$i=2$.]
            We have $\taus{a,a}b = b \in S[0]$ and $\taus{b,a}b \in S[1]$ by \itemref{pr:rules}{ab a}.
            On the other hand, $\taus{c,a}b$ and $\taus{d,a}b$ belong to $\Gamma_2$ and hence to $T[2]$. Hence $S[2] \leq T[2]$.
            
            By \itemref{pr:rules}{a b}, we have $\taus{a,b}c \eqv{1} \taus{a,c}b$, so the $24$ possible elements of $\Gamma_2$ come in pairs that are linearly dependent modulo~$S[1]$. Hence $\dim S[2] \leq 10 + 24/2 = 22$.
        \item[$i=3$.]
            We have $\taus{a,a,b}c = \taus{b}c \in S[1]$, and we have $\taus{b,a,b}c \in S[2]$ by \itemref{pr:rules}{aba c} and $\taus{c,a,b}c \in S[2]$ by \itemref{pr:rules}{abc a}.
            On the other hand, $\taus{d,a,b}c$ belongs to $\Gamma_3$ and hence to $T[3]$. Hence $S[3] \leq T[3]$.
            
            By \itemref{pr:rules}{a b}, we have $\taus{a,b,c}d \eqv{2} \taus{a,b,d}c$, so the $24$ possible elements of $\Gamma_3$ come in pairs that are linearly dependent modulo~$S[2]$. Hence $\dim S[3] \leq 22 + 24/2 = 34$.
        \item[$i=4$.]
            We have $\taus{a,a,b,c}d = \taus{b,c}d \in S[2]$.
            On the other hand, $\taus{b,a,b,c}d$ and $\taus{c,a,b,c}d$ belong to $\Gamma_4$ and hence to $T[4]$. Finally, $\taus{d,a,b,c}d \eqv{3} \taus{d,a,b,d}c \in \Gamma_4$, so $\taus{d,a,b,c}d$ belongs to $\langle S[3], \Gamma_4 \rangle \leq T[4]$.
            Hence $S[4] \leq T[4]$.
            
            By \itemref{pr:rules}{abac d}, the $24$ possible elements of $\{ \taus{a,b,a,c}d \}^G$ come in $8$-tuples that are pairwise linearly dependent modulo~$S[3]$:
            \begin{multline*}
                \taus{a,b,a,c}d \eqv{1} \taus{c,d,c,a}b \eqv{3} \taus{c,d,c,b}a \eqv{1} \taus{b,a,b,c}d \\
                    \eqv{3} \taus{b,a,b,d}c \eqv{1} \taus{d,c,d,b}a \eqv{3} \taus{d,c,d,a}b \eqv{1} \taus{a,b,a,d}c .
            \end{multline*}
            On the other hand, there are no such equivalences between the $24$ possible elements of $\{ \taus{a,b,c,a}d \}^G$. Hence $\dim S[4] \leq 34 + 24/8 + 24 = 61$.
        \item[$i=5$.]
            First, because $\taus{a,b,a,c}d$ is $3$-equivalent to an element beginning with any of the generators $a,b,c,d$, we see that applying any of the four operators $\taus{a}$, $\taus{b}$, $\taus{c}$, $\taus{d}$ on this element will result in an element already contained in $S[3]$.
            
            Next, we apply these operators on $\taus{a,b,c,a}d$. Of course, we again have $\taus{a,a,b,c,a}d \in S[3]$. Next, by \itemref{pr:aba}{aba-bab} and \itemref{pr:rules}{aba c}, we have $\taus{b,a,b,c,a}d \eqv{3} \taus{a,b,a,c,a}d \in S[4]$, and by \itemref{pr:rules}{abcdb a}, we have $\taus{d,a,b,c,a}d \in S[4]$. Finally, $\taus{c,a,b,c,a}d \in \Gamma_5$. Hence $S[5] \leq T[5]$.
            
            By \itemref{pr:rules}{abcab d}, the $24$ possible elements of $\Gamma_5$ come in pairs that are linearly dependent modulo $S[4]$. Hence $\dim S[5] \leq 61 + 24/2 = 73$.
        \item[$i=6$.]
            We have $\taus{a,a,b,c,a,b}d = \taus{b,c,a,b}d \in S[4]$. By \itemref{pr:aba}{aba-bab} and \itemref{pr:rules}{abac d}, we have
            \[ \taus{b,a,b,c,a,b}d \eqv{4} \taus{a,b,a,c,a,b}d \eqv{3} \taus{a,b,b,d,b,a}c \in S[4] . \]
            Next, $\taus{c,a,b,c,a,b}d \in \Gamma_6$, and finally, by \itemref{pr:rules}{abcab d}, we also have $\taus{d,a,b,c,a,b}d \eqv{4} \taus{d,b,a,d,b,a}c \in \Gamma_6$.
            Hence $S[6] \leq T[6]$.
            
            By \cref{pr:finalrule}, the $24$ possible elements of $\Gamma_6$ come in triples that are pairwise linearly dependent module $S[5]$.
            Hence $\dim S[6] \leq 73 + 24/3 = 81$.
        \item[$i=7$.]
            We have $\taus{a,a,b,c,a,b,c}d \in S[5]$, and by \cref{pr:finalrule}, it follows that also $\taus{b,a,b,c,a,b,c}d$ and $\taus{c,a,b,c,a,b,c}d$ belong to $S[5]$.
            Finally, by \cref{pr:keyrule}, we also have $\taus{d,a,b,c,a,b,c}d \in S[6]$.
			We conclude that $S[7]=S[6]$, and therefore $A = S[6]$.
        \qedhere
    \end{itemize}
\end{proof}

\bibliographystyle{alpha}
\bibliography{4-generated}

\end{document}